\documentclass[11pt]{amsart}

\usepackage{enumerate,url,amssymb, mathrsfs}
\usepackage{graphicx}

\newtheorem{theorem}{Theorem}[section]
\newtheorem{lemma}[theorem]{Lemma}
\newtheorem{proposition}[theorem]{Proposition}
\newtheorem*{theorema}{Theorem A}

\theoremstyle{definition}
\newtheorem{definition}[theorem]{Definition}

\newtheorem{conjecture}[theorem]{Conjecture}
\newtheorem{question}[theorem]{Question}
\newtheorem{problem}[theorem]{Problem}

\theoremstyle{remark}
\newtheorem{remark}[theorem]{Remark}

\numberwithin{equation}{section}

\newcommand{\abs}[1]{\lvert#1\rvert}

\newcommand{\C}{\mathbb{C}}
\newcommand{\DD}{\mathbb{D}}
\newcommand{\GR}{\mathcal{G}}
\newcommand{\HH}{\mathbb{H}}

\newcommand{\R}{\mathbb{R}}
\newcommand{\T}{\mathbb{T}}
\newcommand{\TR}{\mathcal{T}}
\newcommand{\Z}{\mathbb{Z}}

\newcommand{\D}{\partial}

\newcommand{\Ha}{\mathcal{H}}
\newcommand{\onto}{\overset{\textnormal{\tiny{onto}}}{\longrightarrow}}
\newcommand{\he}{\overset{{}_\thicksim}{\hookrightarrow}}
\DeclareMathOperator{\cp}{Cap}
\DeclareMathOperator{\diam}{diam}
\DeclareMathOperator{\dist}{dist}
\DeclareMathOperator{\re}{Re}
\DeclareMathOperator{\Mod}{Mod}
\DeclareMathOperator{\im}{Im}

\def\XXint#1#2#3{{\setbox0=\hbox{$#1{#2#3}{\int}$}
\vcenter{\hbox{$#2#3$}}\kern-.5\wd0}}

\def\le{\leqslant}
\def\ge{\geqslant}

\begin{document}

\title[Harmonic mapping problem]{Harmonic mapping problem \\ and affine capacity}

\author{Tadeusz Iwaniec}
\address{Department of Mathematics, Syracuse University, Syracuse,
NY 13244, USA and Department of Mathematics and Statistics,
University of Helsinki, Finland}
\email{tiwaniec@syr.edu}
\thanks{Iwaniec was supported by the NSF grant DMS-0800416 and the Academy of
Finland grant 1128331.}

\author{Leonid V. Kovalev}
\address{Department of Mathematics, Syracuse University, Syracuse,
NY 13244, USA}
\email{lvkovale@syr.edu}
\thanks{Kovalev was supported by the NSF grant DMS-0913474.}

\author{Jani Onninen}
\address{Department of Mathematics, Syracuse University, Syracuse,
NY 13244, USA}
\email{jkonnine@syr.edu}
\thanks{Onninen was supported by the NSF grant DMS-0701059.}


\date{January 12, 2010}


\begin{abstract}
The Harmonic Mapping Problem asks when there exists a harmonic homeomorphism between two given domains.
It arises in the theory of minimal surfaces and in calculus of variations, specifically in hyperelasticity theory.
We investigate this problem for doubly connected domains in the plane, where it already presents considerable challenge and leads to
several interesting open questions.
\end{abstract}

\maketitle

\section{Introduction}\label{intro}

By virtue of Riemann Mapping Theorem for every pair $(\Omega,\Omega^*)$ of simply connected domains in the complex plane one can find
a conformal mapping $h\colon \Omega\onto \Omega^*$ except for two cases: $\Omega\varsubsetneq\C=\Omega^*$ or $\Omega^*\varsubsetneq\C=\Omega$.
The situation is quite different for doubly connected domains.

A domain $\Omega\subset \C$ is doubly connected if $\widehat\C\setminus \Omega$ consists of two connected
components; that is, continua in the Riemann sphere $\widehat\C$. We say that $\Omega$ is nondegenerate
if both components contain more than one point. Every nondegenerate doubly connected domain can be conformally mapped onto an annulus
\begin{equation*}
A(r,R)=\{z\colon r<\abs{z}<R\},\qquad 0<r<R<\infty
\end{equation*}
where the ratio $R/r$ does not depend on the choice of the conformal mapping. This gives rise to the notion of conformal modulus,
\begin{equation}\label{modulus}
\Mod\Omega=\log\frac{R}{r}.
\end{equation}
In fact, modulus is the only conformal invariant for nondegenerate doubly connected domains.  Let us set $\Mod\Omega=\infty$
for the degenerate cases.

Complex harmonic functions, whose real and imaginary parts need not be coupled by the Cauchy-Riemann system, provide significantly larger
class of mappings, but still restrictions on the domains are necessary.
The studies of the Harmonic Mapping Problem began with Rad\'o's theorem~\cite{Ra} which states that there is no harmonic homeomorphism
$h\colon\Omega\onto \C$ for any proper domain $\Omega\varsubsetneq\C$.
The first proof of Rad\'o's theorem in this form was given by Bers~\cite{Be}; both Rad\'o and Bers were motivated by the celebrated
Bernstein's theorem: \emph{any global solution} (in the entire plane) \emph{of the minimal surface equation is affine}.

Harmonic Mapping Problem for doubly connected domains originated from the work of J.C.C. Nitsche on minimal surfaces. In~1962
he formulated a conjecture~\cite{Ni} which was recently proved by the present authors~\cite{IKO2}.

\begin{theorema}
A harmonic homeomorphism $h\colon A(r,R)\to A(r^*,R^*)$  between circular annuli exists
if and only if
\begin{equation}\label{nitbound}
\frac{R^*}{r^*}\ge \frac{1}{2}\left(\frac{R}{r}+\frac{r}{R}\right).
\end{equation}
\end{theorema}

It is a simple matter to see that harmonic functions remain harmonic upon conformal change of the independent variable $z\in\Omega$.
Therefore, Theorem~A remains valid when the annulus $A(r,R)$ is replaced by any doubly connected domain $\Omega\subset\C$
of the same modulus $\log\frac{R}{r}$. The Nitsche bound~\eqref{nitbound} then reads as
\begin{equation}\label{nitbound2}
\frac{R^*}{r^*}\ge \frac{1}{2}\left(e^{\Mod\Omega}+e^{-\Mod\Omega}\right)=\cosh\Mod\Omega.
\end{equation}

The harmonicity of a mapping $h\colon\Omega\to\Omega^*$ is also preserved under affine transformations of the target. Thus it is natural
to investigate necessary and sufficient conditions for the existence of $h$ in terms of the conformal modulus of $\Omega$
and an affine invariant of the target $\Omega^*$. This leads us to the concept of {affine modulus}.

A $\C$-affine automorphism of $\C$ is a mapping of the form $z\mapsto az+c$ with $a,c\in \C$, $a\ne 0$.
An $\R$-affine automorphism of $\C$, or simply affine transformation, takes the form $z\mapsto a z+ b \bar z+c$ with
determinant $\abs{a}^2-\abs{b}^2\ne 0$.

\begin{definition}\label{afmod}
The \emph{affine modulus} of a doubly connected domain $\Omega\subset\C$ is defined by
\begin{equation}\label{afmod1}
\Mod_@\Omega=\sup\{\Mod \phi(\Omega); \quad \phi\colon\C\onto\C \text{ affine}\}.
\end{equation}
\end{definition}
For an equivalent formulation and properties of the affine modulus see Section~\ref{affinesection}.

We can now state our main results: a necessary condition (Theorem~\ref{necessary}) and sufficient condition (Theorem~\ref{sufficient}) for
the existence of a harmonic homeomorphism $h\colon \Omega\onto\Omega^*$.

\begin{theorem}\label{necessary}
If $h\colon \Omega\to\Omega^*$ is a harmonic homeomorphism between doubly connected domains, and $\Omega$ is nondegenerate, then
\begin{equation}\label{affinepsi1}
\frac{\Mod_@ \Omega^*}{\Mod\Omega} \ge \Phi(\Mod \Omega)
\end{equation}
where $\Phi\colon (0,\infty)\to(0,1)$ is an increasing function such that
$\Phi(\tau)\to 1$ as $\tau\to\infty$. More specifically,
\begin{equation}\label{specific}
\Phi(\tau)= \lambda\left(\coth \frac{\pi^2}{2\tau}\right),\quad
{\rm where } \ \lambda(t)\ge \frac{\log t-\log(1+\log t)}{2+\log t}, \ t\ge 1.
\end{equation}
\end{theorem}

\begin{conjecture}\label{genNit} Based upon the Nitsche bound~\eqref{nitbound2} it seems reasonable to
expect that
\[
\Phi(\tau)=\frac{1}{\tau}\log\cosh\tau,\qquad 0<\tau<\infty
\]
but we have no proof of this.
\end{conjecture}

\begin{theorem}\label{sufficient}
Let $\Omega$ and $\Omega^*$ be  doubly connected domains in $\C$ such that
\begin{equation}\label{suff1}
\Mod_@\Omega^* > \Mod\Omega.
\end{equation}
Then there exists a harmonic homeomorphism $h\colon \Omega\to\Omega^*$ unless $\C\setminus\Omega^*$ is bounded.
In the latter case there is no such $h$.
\end{theorem}

When $\Mod\Omega\to\infty$, the comparison of inequalities~\eqref{affinepsi1} and~\eqref{suff1} shows that both
are asymptotically sharp. We do not know if equality in~\eqref{suff1} (with both sides finite) would suffice as well. This is discussed in
Remark~\ref{equalmod}.

We write $\Omega_1\he \Omega_2$ when $\Omega_1$ is contained in $\Omega_2$ in such a way that the inclusion
map $\Omega_1 \hookrightarrow \Omega_2$ is a homotopy equivalence. For doubly connected domains this
simply means that $\Omega_2^-\subset \Omega_1^-$ and $\Omega_1^+\subset \Omega_2^+$.
The monotonicity of modulus can be expressed by saying that $\Omega_1\he \Omega_2$ implies $\Mod\Omega_1\le \Mod\Omega_2$ and
 $\Mod_@\Omega_1\le \Mod_@\Omega_2$.

Observe that both conditions~\eqref{affinepsi1} and~\eqref{suff1} are preserved if $\Omega$ is replaced by a domain with a
smaller conformal modulus, or $\Omega^*$ is replaced by a domain with a greater affine modulus. Thus, Theorems~\ref{necessary}
and~\ref{sufficient} suggest the formulation of the following comparison principles.

\begin{problem}\label{compared} \textsc{(Domain Comparison Principle)}
Let $\Omega$ and $\Omega^*$ be doubly connected domains such that $\Omega$ is nondegenerate and
there exists a harmonic homeomorphism $h\colon\Omega\onto \Omega^*$. If $\Omega_\circ \he \Omega$, then
there exists a harmonic homeomorphism $h_\circ\colon \Omega_\circ\onto \Omega^*$.
\end{problem}

\begin{problem}\label{comparet} \textsc{(Target Comparison Principle)}
Let $\Omega$ and $\Omega^*$ be doubly connected domains such that
there exists a harmonic homeomorphism $h\colon\Omega\onto \Omega^*$. If $\Omega_\circ^*$ is nondegenerate and
$\Omega^* \he \Omega_\circ^*$, then there exists a harmonic homeomorphism $h_\circ\colon \Omega\onto \Omega_\circ^*$.
\end{problem}

Both Problems~\ref{compared} and~\ref{comparet} are open to the best of our knowledge. Although the Harmonic Mapping
Problem has its roots in the theory of minimal surfaces, it also arises from the minimization of the Dirichlet energy
\[
\mathcal E[f]= \iint_\Omega \abs{Df}^2
\]
among all homeomorphism $f\colon \Omega\onto\Omega^*$. The minimizers and stationary points of $\mathcal E$
serve as a model of admissible deformations of a hyperelastic material with stored energy $\mathcal E$~\cite{MHb}.
The existence of a harmonic homeomorphism $h\colon\Omega\onto \Omega^*$ is necessary for the minimum
of $\mathcal E[f]$ to be attained in the class of homeomorphisms between $\Omega$ and $\Omega^*$. Indeed, such
minimal maps must be harmonic.

The rest of the paper is organized as follows: we discuss relevant affine invariants in Section~\ref{affinesection},
prove Theorem~\ref{necessary} in Section~\ref{necessarysection}, and prove Theorem~\ref{sufficient}
in Section~\ref{sufficientsection}. We write $\DD_r=\{z\in\C\colon\abs{z}<r\}$, $\T_r=\partial \DD_r$, and $\T=\T_1$.

The Harmonic Mapping Problem for multiply connected domains has been also
studied in~\cite{BH93,BH94,BH99,Dub,DH97,HS85,Lyz01},
where the authors mostly focus on existence of harmonic mappings onto some domain of a given canonical type (such as disk with punctures)
rather than onto a specific domain.

\section{Affine capacity and modulus}\label{affinesection}

In this section $\Omega$ is a doubly connected domain in $\C$, possibly degenerate.
Exactly one of the components of $\C\setminus \Omega$ is bounded and is denoted $\Omega^-$.
We also write $\Omega^+=\Omega^-\cup\Omega$, which is a simply connected domain.
Let $\abs{E}$ denote the area (planar Lebesgue measure) of a set $E\subset \C$.
We are now ready to introduce the first affine invariant of $\Omega$.

\begin{definition}\label{carmod}
The \emph{Carleman modulus} of $\Omega$ is defined by the rule
\[
\Mod_{{}_\copyright}\Omega =\frac{1}{2}\log\frac{\abs{\Omega^+}}{\abs{\Omega^-}}
=\frac{1}{2}\log\left(1+\frac{\abs{\Omega}}{\abs{\Omega^-}}\right)
\]
unless $\abs{\Omega}=\infty$ or $\abs{\Omega^-}=0$, in which case $\Mod_{{}_\copyright}\Omega:=\infty$.
\end{definition}

The well-known inequality
\begin{equation}\label{carleman0}
\Mod\Omega \le \Mod_{{}_\copyright} \Omega,
\end{equation}
proved by T.~Carleman~\cite{Ca18} in~1918, can be  expressed by saying that among all doubly
connected domains with given areas of $\Omega^-$ and $\Omega^+$ the maximum of conformal modulus
is uniquely attained by the circular annulus $A(r,R)$,
$\pi r^2=\abs{\Omega^-}<\abs{\Omega^+}=\pi R^2$. Carleman's inequality was one of the earliest
isoperimetric-type results in mathematical physics, many of which can be found in~\cite{PSb}.

Our second affine invariant arises from an energy minimization problem. Recall the capacity of $\Omega$,
\begin{equation}\label{cap}
\cp \Omega=\inf_u \iint_{\C} \abs{\nabla u}^2
\end{equation}
where the infimum is taken over all real-valued smooth functions on $\C$ which assume precisely two values $0$ and $1$
on $\C\setminus\Omega$. The conformal modulus, defined in~\eqref{modulus}, is given by
\[
\Mod \Omega=\frac{2\pi}{\cp \Omega}.
\]

Since we are looking for an affine invariant of $\Omega$, the following variational problem is naturally introduced.
\begin{definition}\label{affmodulus}
Define the \emph{affine capacity} of $\Omega$ by
\[\cp_@ \Omega := \inf_{A,u} \frac{1}{\abs{\det A}}\iint_{\Omega} \abs{A\nabla u}^2\]
where the infimum is taken over all invertible matrices $A$ and over all real functions $u$ as in~\eqref{cap}. Thus,
the affine modulus of $\Omega$, defined in~\eqref{afmod1}, is
\[
\Mod_@ \Omega=\frac{2\pi}{\cp_@ \Omega}.
\]
\end{definition}

Let us now examine the properties of the affine modulus. From~\eqref{carleman0} we see that
\begin{equation}\label{carleman01}
\Mod\Omega \le \Mod_@\Omega \le \Mod_{{}_\copyright} \Omega.
\end{equation}
When $\Omega$ is a circular annulus, equality holds in~\eqref{carleman0} and therefore in~\eqref{carleman01}.
Hence $\Mod_@ A(r,R)=\log (R/r)$.

Equality $\Mod_@\Omega=\Mod\Omega$ is also attained, for example, if $\Omega$ is the \emph{Teichm\"uller ring}
\[\TR(s):=\C\setminus ([-1,0]\cup [s,+\infty)),\qquad s>0\]
Indeed, for any affine automorphism $\phi\colon\C\to\C$ there is a $\C$-affine automorphism $\psi\colon\C\to\C$
that agrees with $\phi$ on $\R$. Since $\phi(\TR(s))=\psi(\TR(s))$ and $\psi$ is conformal, it follows that
\begin{equation}\label{teichstays}
\Mod \phi(\TR(s)) = \Mod \TR(s).
\end{equation}

On the other extreme, $\Mod_@ \Omega$ may be infinite even when $\Mod \Omega$ is finite. Indeed, the domain
\[\Omega=\{z\in \C\colon \abs{\im z}<1\}\setminus [-1,1]\]
is affinely equivalent to $\{z\in \C\colon \abs{\im z}<1\}\setminus [-s,s]$ for any $s>0$.
The conformal modulus of the latter domain grows indefinitely as $s\to 0$.

Since the second inequality in~\eqref{carleman01} becomes vacuous when $\abs{\Omega}=\infty$ or $\abs{\Omega^-}=0$,
it is desirable to have an upper estimate for $\Mod_@\Omega$ in terms of other geometric properties of $\Omega$.
Recall that the \emph{width} of a compact set $E\subset \C$, denoted $w(E)$, is the smallest distance between
two parallel lines that enclose the set. For connected sets this is also the length of the shortest $1$-dimensional projection.

\begin{lemma}\label{finiteaffmod}
Let $\Omega$ be a nondegenerate doubly connected domain. If $w=w(\Omega^-)>0$, then
\begin{equation}\label{fam1}
\Mod_@ \Omega \le \Mod \TR(d/w),\qquad \text{where } d=\dist(\partial \Omega^+,\Omega^-).
\end{equation}
\end{lemma}

\begin{proof} Let $\phi\colon\C\to\C$ be an affine automorphism. Denote its Lipschitz constant by $L:=\abs{\phi_z}+\abs{\phi_{\bar z}}$.
Clearly $\dist(\partial \phi(\Omega^+),\phi (\Omega^-))\le L\,d$ and $\diam \phi(\Omega^-)\ge L w$.
Now~\eqref{fam1} follows from the extremal property of the Teichm\"uller ring: it has the greatest conformal modulus among all domains with
given diameter of the bounded component and given distance between components~\cite[Ch. III.A]{Ahb}.
\end{proof}

Even when the affine modulus is finite, the supremum in~\eqref{afmod1} is not always attained.
An example is given by the \emph{Gr\"otzsch ring}
\[\GR(s)=\{z\in \C\colon \abs{z}>1\}\setminus [s,+\infty),\qquad s>1.\]
Indeed,
\begin{equation}\label{gtmod}
\Mod_@ \GR(s)\le \Mod \mbox{$\TR\big(\frac{s-1}{2}\big)$}
\end{equation}
by Lemma~\ref{finiteaffmod}.
Equality holds in~\eqref{gtmod} because the images of $\GR(s)$ under mappings of the form $z+k\bar z$, $k\nearrow 1$,
converge to the domain $\C\setminus ([-2,2]\cup [2s,\infty))$ which is a $\C$-affine image of
$\TR\big(\frac{s-1}{2}\big)$. Yet, for any affine automorphism $\phi$
\[
\Mod \phi (\GR(s))< \Mod \phi\big\{\C\setminus ([-1,1]\cup [s,\infty))\}=\Mod \mbox{$\TR\big(\frac{s-1}{2}\big)$}
\]
where the first part expresses the monotonicity of modulus, and the second follows from~\eqref{teichstays}.
Thus the supremum in~\eqref{afmod1} is not attained by any $\phi$.

\begin{remark}\label{equalmod}
Since equality holds in~\eqref{gtmod}, Lemma~\ref{finiteaffmod} is sharp. Furthermore, the pair of
domains $\Omega=\TR\big(\frac{s-1}{2}\big)$ and $\Omega^*=\GR(s)$ can serve as a test case for whether
equality in~\eqref{suff1} implies the existence of a harmonic homeomorphism.
\end{remark}

\begin{remark}
The identity
\begin{equation}\label{relmod}
\Mod_@ \GR(s)=\Mod_@ \mbox{$\TR\big(\frac{s-1}{2}\big)$}
\end{equation}
somewhat resembles the relation between conformal moduli of the Gr\"otzsch and
Teichm\"uller rings~\cite[Ch. III.A]{Ahb},
\[\Mod \GR(s)=\frac{1}{2}\Mod \TR(s^2-1).\]
\end{remark}

\section{Proof of Theorem~\ref{necessary}}\label{necessarysection}

Before proceeding to the proof we recollect basic facts of potential theory in the plane which can be found in~\cite{Ranb}.
A domain $\Omega$ has Green's function $G_{\Omega}$ whenever $\C\setminus \Omega$ contains
a nondegenerate continuum. Our normalization is $G_{\Omega}(z,\zeta)=-\log\abs{z-\zeta}+O(1)$ as $z\to\zeta$.
In particular, $G_{\Omega}(z,\zeta)>0$. Green's function for the unit disk is
\begin{equation}\label{gdisk}
G_{\DD}(z,\zeta)=\log \left|\frac{1-z \bar \zeta}{z-\zeta}\right|.
\end{equation}
If $f\colon \Omega\to\Omega^*$ is a holomorphic function, then the subordination principle holds:
\begin{equation}\label{gsub}
G_{\Omega}(z,\zeta)\le G_{\Omega^*}(f(z),f(\zeta)).
\end{equation}
One consequence of~\eqref{gdisk} and~\eqref{gsub} is a general version of the Schwarz lemma.
If $f\colon \Omega\to\DD$ is holomorphic and $f(a)=0$ for some $a\in\Omega$, then
\begin{equation}\label{schw}
\abs{f(z)}\le \exp(-G_{\Omega}(z,a))<1,\qquad z\in\Omega.
\end{equation}

There is an application of~\eqref{schw} to harmonic homeomorphisms.
We refer to~\cite[p.5]{Dub} for a discussion of relation between harmonic and quasiconformal mappings, and to the
book~\cite{Ahb} for the general theory of quasiconformal mappings.

\begin{proposition}\label{distprop} Let $\Omega\subset \C$ be a domain with Green's function $G_{\Omega}$. Fix $a\in \Omega$.
For any harmonic homeomorphism $h\colon \Omega\to \C$ there exists an affine automorphism
$\phi$ such that the composition $H=\phi\circ h$ satisfies
\begin{equation}\label{disth}
\frac{\abs{H_{\bar z}}}{\abs{H_z}} \le \exp(-G_{\Omega}(z,a))<1
\end{equation}
for $z\in \Omega$.
\end{proposition}

\begin{proof} Replacing $h$ with $\bar h$, we may assume that $h$ is orientation preserving. Then $h$
satisfies the second Beltrami equation
\begin{equation}\label{disth2}
h_{\bar z}(z)=\nu (z)\overline{h_z(z)}
\end{equation}
where the \emph{second complex dilatation} $\nu$ is an antiholomorphic function from $\Omega$ into $\D$.
Since $\abs{\nu(z)}<1$, the mapping $h$ is quasiconformal away from the boundary of $\Omega$.
Affine transformations of $h$ do not affect harmonicity by may decrease $\abs{\nu}$.
Indeed, the composition
\[H(z)=h(z)-\varkappa \bar h(z),\qquad \varkappa\in\DD, \ \alpha\in\C\setminus\{0\}\]
satisfies new second Beltrami equation
\begin{equation}\label{disth4}
H_{\bar z}(z)=\frac{\nu(z)-\varkappa}{1-\overline{\varkappa} \nu(z)}\overline{H_z(z)}
\end{equation}
By setting $\varkappa =\nu(a)$ we achieve that the second complex dilatation of $H$ vanishes at $a$.
Now the required estimate~\eqref{disth} follows from~\eqref{schw}.
\end{proof}

We would like to make the estimate~\eqref{disth} more explicit
when the domain $\Omega$ is a circular annulus $\Omega=A(R^{-1},R)$.
Although an explicit formula for Green's function of an annulus can be found, for us it suffices to have a lower bound.
We obtain such a bound from the subordination principle~\eqref{gsub}. Indeed, for any $\alpha>0$
the vertical strip $S=\{z\in\C\colon \abs{\re z}<\frac{\pi}{2\alpha}\}$ has Green's function~\cite[p.~109]{Ranb}
\begin{equation}\label{gg1}
G_S(z,\zeta)=\log\left|\frac{e^{i\alpha z}+e^{-i\alpha \bar \zeta}}{e^{i\alpha z}-e^{i\alpha \zeta}}\right|.
\end{equation}
We use~\eqref{gg1} with $\zeta=0$ and $z=iy$ where $\abs{y}\le \pi$
\begin{equation}\label{gg2}
G_S(iy,0)=\log\left|\frac{e^{-\alpha y}+1}{e^{-\alpha y}-1}\right|
=\log \coth \frac{\alpha \abs{y}}{2} \ge \log \coth \frac{\pi \alpha}{2}.
\end{equation}
Since $w=e^z$ maps $S$ onto $\Omega=A(R^{-1},R)$ with $\Mod \Omega= 2\log R=\pi/\alpha$,
inequality~\eqref{gg2} together with the subordination principle~\eqref{gsub} yield
\begin{equation}\label{gann}
G_{\Omega}(z,\zeta)\ge \log \coth \frac{\pi^2}{4\log R},\qquad z,\zeta\in \T.
\end{equation}
We are now ready to give an explicit estimate for holomorphic functions with at least one zero in an annulus.

\begin{lemma}\label{slem}
 Let $\Omega=A(R^{-1},R)$ and suppose that $f\colon \Omega\to\DD$ is a holomorphic function with $f(1)=0$.
Then for each $0\le \alpha<1$ we have
\begin{equation}\label{slem1}
\max_{R^{-\alpha}\le \abs{z}\le R^\alpha} \abs{f(z)} \le k^{1-\alpha},\qquad k=\tanh \frac{\pi^2}{4\log R}<1
\end{equation}
\end{lemma}

\begin{proof} The case $\alpha=0$ immediately follows from~\eqref{schw} and~\eqref{gann}. For the general case, let
\[
M(r)=\max_{\abs{z}=r}\abs{f(z)},\qquad \frac{1}{R}< r <R
\]
By Hadamard's three circle theorem, $\log M(r)$ is a convex function of $\log r$. Since $\log M(1)\le \log k$
and $\log M(r)<1$ for all $r$, the convexity implies
\[\log M(R^{\alpha})\le (1-\alpha)\log k\]
and the same estimate holds for $\log M(R^{-\alpha})$.
\end{proof}

\begin{proof}[Proof of Theorem~\ref{necessary}]
With the aid of conformal transformation of $\Omega$ we may assume that $\Omega=A(R^{-1},R)$
where $2\log R=\Mod\Omega$. As in the proof of Proposition~\ref{distprop}, we apply an affine transformation
to obtain a harmonic mapping $H$ whose second complex dilatation vanishes at~$1$. By Lemma~\ref{slem}
the restriction of $H$ to $\Omega_{\alpha}=A(R^{-\alpha},R^{\alpha})$ is $K$-quasiconformal; that is,
\[
\frac{\abs{H_z}+\abs{H_{\bar z}}}{\abs{H_z}-\abs{H_{\bar z}}} \le
K=\frac{1+k^{1-\alpha}}{1-k^{1-\alpha}},\qquad k=\tanh \frac{\pi^2}{4\log R}.
\]
The conformal modulus of $\Omega_{\alpha}$ is distorted by the factor of at most $K$ under the mapping $H$, see~\cite{Ahb}.
Hence
\begin{equation}\label{psi021}
\Mod_{@}h(\Omega)\ge \Mod H(\Omega)\ge \frac{1-k^{1-\alpha}}{1+k^{1-\alpha}}\Mod \Omega_{\alpha}
= \alpha\,\frac{1-k^{1-\alpha}}{1+k^{1-\alpha}}\Mod \Omega.
\end{equation}
We are free to choose any $0<\alpha<1$ in~\eqref{psi021}. Introduce the function
\begin{equation}\label{lam1}
\lambda(t)=\sup_{0<\alpha<1} \alpha\,\frac{t^{1-\alpha}-1}{t^{1-\alpha}+1},\qquad t\ge 1
\end{equation}
which is positive and decreasing for $t>0$. Clearly $\lambda(t)\to 1$ as $t\to\infty$.
Now~\eqref{psi021} takes the form
\begin{equation}\label{psi022}
\frac{\Mod_{@}h(\Omega)}{\Mod\Omega}\ge \Phi(\Mod\Omega),\qquad \text{where }\Phi(\tau)=\lambda\Big(\coth \frac{\pi^2}{2\tau}\Big)
\end{equation}
and $\Phi(\tau)\to 1$ as $\tau\to\infty$. To obtain a concrete bound, we test the supremum in~\eqref{lam1} by putting
\[
\alpha=1-\frac{\log(1+\log t)}{\log t},
\]
obtaining
\[
\lambda(t)\ge \left(1-\frac{\log(1+t)}{t}\right)\frac{\log t}{2+\log t}
\]
which is~\eqref{specific}.
\end{proof}

\section{Proof of Theorem~\ref{sufficient}}\label{sufficientsection}

For notational simplicity we denote by $\Ha(\Omega,\Omega^*)$ the class of harmonic homeomorphisms from $\Omega$ onto $\Omega^*$.
This includes conformal mappings $\Omega\to\Omega^*$ if they exist. The proof of Theorem~\ref{sufficient} is divided into three parts.

\subsection{Exceptional Pairs}\label{except}

In this section we prove that $\Ha (\Omega,\Omega^*)$ is empty when $\Mod\Omega<\infty$ and $\C\setminus\Omega^*$ is bounded.
Suppose to the contrary that $h\in\Ha (\Omega,\Omega^*)$. Up to a conformal
transformation, $\Omega$ is the circular annulus $A(1,R)$.
The target domain is  $\Omega^*=\C\setminus E$ for some compact set $E$.
With the help of inversion $z\mapsto R^2/z$ we can arrange so that $h(z)$ approaches $E$ as $\abs{z}\to 1$.

For large enough integers $m$ the open disk $\DD_m$ contains $E$. Let $\Omega_m=h^{-1}(\DD_m\setminus E)$.
As $m\to\infty$, the outer boundary of $\Omega_m$ approaches $\T_R$.
Thus there exist conformal mappings $g_m\colon A(1,R_m)\onto \Omega_m$ such that $g_m(z)\to z$ pointwise
and $R_m\nearrow R$ as $m\to\infty$. Define a harmonic mapping
\[h_m(z)=\frac{1}{m}h(g_m(R_m z)),\qquad z\in A(R_m^{-1},1)\]
and observe that for any $z\in A(R^{-1},1)$ the value $h_m(z)$ is defined when $m$ is large enough.
Moreover, $h_m(z)\to 0$ because $g_m(R_m z)\to Rz$, and the convergence is uniform on compact subsets of $A(R^{-1},1)$, i.e.,
\begin{equation}\label{convto0}
\lim_{m\to\infty} \sup_{r_1\le \abs{z}\le r_2}\abs{h_m(z)}\to 0,\qquad R^{-1}<r_1<r_2<1.
\end{equation}
Let us write
\begin{equation}\begin{split}
h_m(z)&=\sum_{n\ne 0}\left(a_n z^n+b_n \bar z^{-n}\right) + a_0 \log\abs{z}+b_0 \\
&=\sum_{n\ne 0}\left(a_n \rho^n+b_n \rho^{-n}\right)e^{ni\theta} + a_0 \log \rho+b_0,\quad z=\rho e^{i\theta}
\end{split}
\end{equation}
where the coefficients depend on $m$ as well but we suppress this dependence in the notation. In fact,~\eqref{convto0} implies that
for each fixed $n$ the coefficients $a_n,b_n$ tend to $0$ as $m\to\infty$.
On the other hand, $h_m$ extends to a sense-preserving homeomorphism $\T\to\T$, which leads to a contradiction with the
following result. \qed

\begin{theorem}\label{weitsman} (Weitsman~\cite{We98})
Let $f\colon\T\to\T$ be a sense preserving homeomorphism with Fourier coefficients $c_n$, $n\in\Z$. Then
\begin{equation}\label{weitsman1}
\abs{c_0}+\abs{c_1}\ge \frac{2}{\pi}.
\end{equation}
\end{theorem}

Weitsman's inequality is sharp. Earlier Hall~\cite{Ha82} proved~\eqref{weitsman1} with $1/2$ instead of $2/\pi$.
The validity of~\eqref{weitsman1} with some absolute constant can be traced to the work of several
authors, see~\cite{Ha83}. It is worth mentioning that~\eqref{weitsman1} first arose as a special case of
Shapiro's conjecture~\cite[Problem 5.41]{ABB}, which posed that for any sense-preserving $k$-fold cover $f\colon\T\to \T$
\begin{equation}\label{shapiro}
\abs{c_0}^2+\abs{c_1}^2+\dots+\abs{c_k}^2 \ge \delta_k
\end{equation}
where $\delta_k>0$ depends only on $k$. This conjecture was proved by Hall~\cite{Ha82,Ha83} for $k=2$
and by Sheil-Small~\cite{SS} in general. The rate of decay of $\delta_k$ remains unknown, see~\cite{Ha86}.

\subsection{Non-exceptional pairs with degenerate target}\label{degen}

Here we assume that $\Mod \Omega^*=\infty$ but $\C\setminus\Omega^*$ is unbounded.
Without loss of generality, $\Omega^*=G\setminus\{0\}$ where $G\varsubsetneq\C$ is a simply connected domain and $0\in G$.
For $t\ge 0$ we define a mapping $F_t\colon\C\setminus \{0\}\to\C$ by
\[F_t(re^{i\theta}) = (r+\sqrt{r^2+t^2}) e^{i\theta} \]
Note that $\abs{F_t(z)}\ge t$ for all $z\ne 0$.
The choice of $F_t$ is motivated by the fact that its inverse is harmonic:
\[F_t^{-1}(\zeta)=\frac{1}{2}\left(\zeta-\frac{t^2}{\bar \zeta}\right) \]
Since $F_t^{-1}$ maps the doubly connected domain $F_t(\Omega^*)$ onto $\Omega^*$, it remains
to find $t$ such that $\Mod F_t(\Omega^*)=\Mod \Omega$. The latter will follow from the intermediate value theorem
once we prove $\Mod F_t(\Omega^*)\to 0$ as $t\to\infty$ and $\Mod F_t(\Omega^*)\to \infty$ as $t\to 0$.
The latter is obvious, so we proceed to the proof of the former limit.

Let $d=\dist(\partial G,0)$. The complement of $F_t(\Omega^*)$ has two components: one is the disk $\overline{\DD_t}$
and the other contains a point with absolute value $d+\sqrt{d^2+t^2}$. By the extremal property of the Gr\"otzsch ring~\cite[Ch. III.A]{Ahb},
$\Mod F_t(\Omega^*)$ does not exceed the conformal modulus of the Gr\"otzsch ring $\GR(s)$ with
\[s=\frac{d+\sqrt{d^2+t^2}}{t}\]
As $t\to\infty$, we have $s\to 1$ and thus $\Mod \GR(s)\to 0$. This completes the proof. \qed

\subsection{Non-exceptional pairs with nondegenerate target}\label{nondegen}

Since harmonicity is invariant under affine transformations of the target, we may assume that
$\Mod\Omega^*>\Mod\Omega$. There are two substantially different cases.

\textbf{Case 1.} The set $\C\setminus \Omega^*$ is contained in a line. Thus, up to a $\C$-affine automorphism, $\Omega^*$ is the
Teichm\"uller ring
\[\TR(t)=\C\setminus ([-1,0]\cup [t,+\infty))\]
We may and do replace $\Omega$ with a conformally equivalent domain $\TR(s)$ for some $0<s<t$.
Thus our task is to harmonically map $\TR(s)$ onto $\TR(t)$.

Let $b\ge 0$ be a number to be chosen later. Define a piecewise linear function $g\colon\R\to\R$ by
\[
g(x)=\begin{cases} b, \quad &x\le -1 \\
-bx, \quad & -1\le x\le0 \\
0,\quad & x\ge 0 \\
\end{cases}
\]
and consider the domain $G_b=\{x+iy\colon y>g(x)\}$ shown in Figure~1.

\begin{figure}[!h]
\begin{center}
\includegraphics*[height=1.6in]{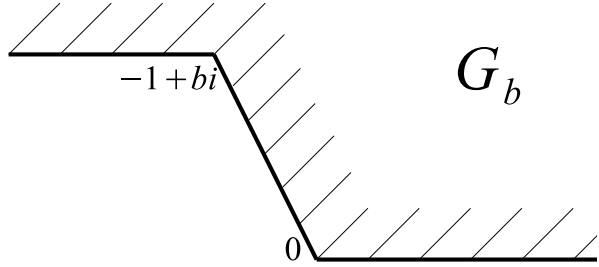}
\caption{Polygonal domain $G_b$}
\end{center}
\end{figure}

Let $f=f_b$ be a conformal mapping of the upper halfplane $\HH=\{z\colon \im z>0\}$
onto $G_b$ normalized by boundary conditions $f(-1)=-1+bi$, $f(0)=0$, and $f(\infty)=\infty$.
It is important to notice that the boundary of $G_b$ satisfies the quasiarc condition uniformly with respect to $b$; that is,
\begin{equation}\label{qarc}
\left|\frac{\zeta_2-\zeta_1}{\zeta_2-\zeta_1}\right|\le C
\end{equation}
for any three points on $\partial G_b$ such that $\zeta_3$ separates $\zeta_1$ and $\zeta_2$.
By a theorem of Ahlfors~\cite[p.~49]{Ahb} $f$ extends to a $K$-quasiconformal mapping $\C\to \C$ with $K$ independent of $b$.
The latter can be expressed via the quasisymmetry condition (see~\cite{TV} or~\cite[Ch.~11]{Heb}):
there is a homeomorphism $\eta\colon [0,\infty)\to [0,\infty)$, independent of $b$, such that
\begin{equation}\label{qs1}
\frac{\abs{f(p)-f(q)}}{\abs{f(p)-f(r)}}\le \eta \left(\frac{\abs{p-q}}{\abs{p-r}}\right)
\end{equation}
for all distinct points $p,q,r\in \C$. Applying~\eqref{qs1} to the triple $-1,0,s$, we find
\begin{equation}\label{qs2}
f(s)\ge \frac{\sqrt{1+b^2}}{ \eta(1/s)}
\end{equation}
Hence $f(s)\to \infty$ as $b\to\infty$. On the other hand, $f(s)=s$ when $b=0$
Since $f(s)=f_b(s)$ depends continuously on $b$, there exists $b>0$ for which $f(s)=t$. Let us fix such $b$.

As observed above, $f_b$ has a continuous extension to $\overline{\HH}$. It takes the segments
$(-\infty,-1)$ and $(0,s)$ into horizontal segments $(-\infty,-1)+ib$ and $(0,t)$ respectively.
By the reflection principle $f_b$ extends holomorphically across each segment, and we have $\re f(\bar z)=\re f(z)$.
It follows that the function $u(z)=\re f(z)$ extends harmonically to the entire Teichm\"uller ring $\TR(s)$.

Consider the harmonic mapping
\begin{equation}\label{defh}
h(z)=u(z)+i\im z
\end{equation}
which by construction is continuous in $\C$. We claim that $h$
is a homeomorphism from $\TR(s)$ onto $\TR(t)$. Since $h$ agrees with $f$ on $\R$, it follows that $h$ maps $\R$ homeomorphically onto $\R$
in such a way that $h(-1)=-1$, $h(0)=0$, and $h(s)=t$. To prove that $h$ is injective in $\TR(s)$, we only need to show that $u(x,y)$ is a strictly
increasing function of $x$ for any fixed $y>0$. The partial derivative $u_x$ is harmonic and nonconstant in $\HH$, has nonnegative boundary
 values on $\R$, and is bounded at infinity. Thus $u_x>0$ in $\HH$ by the maximum principle~\cite{Ranb} and the claim is proved.

\textbf{Case 2.} The set $\C\setminus \Omega^*$ is not contained in a line. We claim that there exists an affine automorphism
$\phi$ such that $\Mod\phi(\Omega^*)=\Mod\Omega$. If this holds, then $\Omega $ can be mapped conformally onto $\phi(\Omega^*)$ and
the composition with $\phi^{-1}$ furnishes the desired harmonic homeomorphism.
Since $\Mod\Omega^*>\Mod\Omega$  and $\Mod\phi(\Omega^*)$ depends continuously on the coefficients of $\phi$, it suffices to
show
\begin{equation}\label{suffices}
\inf_{\phi} \{\Mod\phi(\Omega^*);\quad \phi\colon\C\onto\C \text{ affine}\} =0
\end{equation}
Let $E$ and $F$ be the bounded and unbounded components of the complement of $\Omega^*$, respectively. We need a lemma.

\begin{lemma}\label{trunk} There exists a compact set $\tilde F\subset F$ such that the union $E\cup \tilde F$
is not contained in a line.
\end{lemma}
\begin{proof} Let $F_R=\{z\in F\colon \abs{z}\le R\}$ for $R>0$. If $E$ is not contained in any line,
then we can let $\tilde F$ be any nontrivial connected component of $F_R$, with $R$ large.
If $E$ is contained in some line $\ell$, then $F_R$ is not contained in $\ell$ when $R$ is large enough, and we can again choose
$\tilde F$ to be one of its connected components.
\end{proof}

There exist two parallel lines $\ell_1$ and $\ell_2$ such that any line between $\ell_1$ and $\ell_2$ meets both $E$ and $\tilde F$.
We may assume that these lines are vertical, say $\re z=0$ and $\re z=a$. Let $D=\diam(E\cup \tilde F)$.
For each $0<t<a$ there is a segment of line $\re z=t$ that joins $E$ to $F$ and has length is at most $D$.
The extremal length of the family of all such segments is at most $D/a$, see~\cite[Ch.~I]{Ahb}.
Under the affine transformation $\phi(x+iy)=Mx+iy$ the length of vertical segments is unchanged, while the width
of their family becomes $Ma$ instead of $a$. Thus the extremal length of all rectifiable curves connecting
$\phi(E)$ and $\phi(\tilde F)$ tends to $0$ as $M\to\infty$. The relation between modulus and extremal length~\cite[Ch.~I]{Ahb}
implies $\Mod\phi(\Omega^*)\to 0$ as $M\to\infty$. This completes the proof of~\eqref{suffices} and of Theorem~\ref{sufficient}.
\qed

\section{Concluding remarks}\label{conclude}

Since Conjecture~\ref{genNit} is known to be true for circular annuli~\cite{IKO2} (Nitsche Conjecture),
it is natural to test it on other canonical doubly connected domains: the Gr\"otzsch and Teichm\"uller rings. Precisely, the
questions are as follows.
\begin{question}
\textit{Gr\"otzsch--Nitsche Problem}: for which $1<s,t<\infty$ does there exist a harmonic homeomorphism
$h\colon \GR(s)\onto \GR(t)$?
\newline \textit{Teichm\"uller--Nitsche Problem}: for which $0<s,t<\infty$ does there exist a harmonic
homeomorphism $h\colon \TR(s)\onto \TR(t)$?
\end{question}

We offer the following observation.

\begin{remark}\label{tex} There exists a harmonic homeomorphism $h\colon \TR(s)\onto \TR(t)$ provided that
\begin{equation}\label{tex1}
\frac{t+1}{t} \le \Big(\frac{s+1}{s}\Big)^{3/2}.
\end{equation}
\end{remark}

Indeed, by virtue of Theorem~\ref{sufficient} we only need to consider $t<s$. Choose $1<\alpha \le 3/2$ so that
\[\frac{t+1}{t} = \Big(\frac{s+1}{s}\Big)^{\alpha}.\]
Let $G= \C\setminus (-\infty,0]$. Choose a branch of $f(z)=z^{\alpha}$ in $G$ so that $f(1)=1$. Note that $\re f'>0$ in $G$.
Therefore, the harmonic function $u(z)=\re f(z)$ satisfies $u_x>0$ in $G$. It follows that $h(z):=u(z)+i\im z$ is a harmonic
homeomorphism of $G$ onto itself. It remains to observe that
$h$ maps the domain $G\setminus [s,s+1]$ onto $G\setminus [s^\alpha,(s+1)^\alpha]$. \qed

Numerical computations with conformal moduli of Teichm\"uller rings~\cite{Ahb,AVV} show that Example~\ref{tex} does not
contradict Conjecture~\ref{genNit}.

\bibliographystyle{amsplain}

\end{document}